\documentclass[11pt,reqno]{amsart}
\usepackage{amsmath,amsthm,amscd,amsfonts,amssymb,color}
\usepackage{cite}
\usepackage[mathscr]{eucal}
\usepackage[bookmarksnumbered,colorlinks,plainpages]{hyperref}
\newtheorem{theorem}{Theorem}[section]

\newtheorem{corollary}[theorem]{Corollary}
\theoremstyle{definition}
\newtheorem{definition}[theorem]{Definition}

\newtheorem{Open Prob}[theorem]{Open Problem}
\theoremstyle{remark}
\newtheorem{remark}[theorem]{Remark}
\numberwithin{equation}{section}
\def\DJ{\leavevmode\setbox0=\hbox{D}\kern0pt\rlap{\kern.04em\raise.188\ht0\hbox{-}}D}
\begin{document}

\title[Metrizability of $b$-metric space and $\theta$-metric space]{Metrizability of $b$-metric space and $\theta$-metric space via Chittenden's metrization theorem}

\author[S.\ Som]
{Sumit Som}

\address{ Sumit Som,
                    Department of Mathematics,
                    National Institute of Technology
                    Durgapur, India.}
                    \email{somkakdwip@gmail.com}


\begin{abstract}
In \cite[\, An, V.T., Tuyen, Q.L., Dung, V.N., Stone-type theorem on $b$-metric spaces and applications, Topology Appl. 185-186 (2015) 50-64]{an}, Tran Van An et al. provide a sufficient condition for $b$-metric space to be metrizable. They proved the metrizability by assuming that the distance function is continuous in one variable. The main purpose of this manuscript is to provide a direct short proof of the metrizability of $b$-metric space introduced by Khamsi and Hussain in \cite[\, Khamsi, M.A and Hussain, N., KKM mappings in metric type spaces, Nonlinear Anal. 73 (9) (2010) 3123-3129]{kh} via Chittenden's metrization theorem without any assumption on the distance function. Further in this short note, we prove the metrizability of $\theta$-metric space introduced by Khojasteh et al. in \cite[\, Khojasteh, F., Karapinar, E., Radenovic, S., $\theta$-metric space: A Generalization, Mathematical problems in Engineering, Volume 2013, Article 504609, 7 pages]{ks}.

\end{abstract}

\maketitle

\setcounter{page}{1}

\section{\bf Introduction}
\baselineskip .55 cm
Many mathematicians are attracted to work on a topic which is more fundamental and has lot of applications in many diversified fields. One of the major motivations is to generalize or weaken a certain structure and develop new results compatible to the weaker one. Indeed, there are generalizations which  genuinely develop the subject as a whole and also, there are some, which contribute nothing new to the literature.
Likewise, to generalize the notion of distance functions, in 1906, Fr$\acute{e}$chet first introduced the concept of metric spaces and a century later, we have numerous generalizations of the metric structure.  Unfortunately, some of the generalizations become redundant and turn into metrizable merely adding premise to the subject. As in 2007, Huang and Zhang \cite{HZ} introduced the notion of a cone metric space on a positive cone in a Banach space. Following that, a lot of research articles dealt with the setting and evolved the structure with a number of results. Although in 2011, Khani and Pourmahdian \cite{KP} explicitly constructed a metric on a specified cone metric space and proved that cone metric spaces are metrizable. 

In the year 1993, Czerwik \cite{cz1} introduced the notion of $b$-metric as a generalization of a metric and also in the year 1998, Czerwik \cite{cz2} modified this notion where the coefficient $2$ was replaced by coefficient $s\geq 1.$ In the year 2010, Khamsi and Hussain \cite{kh} defined the concept of a $b$-metric under the name metric-type spaces. To avoid confusion, the metric type in the sense of Khamsi and Hussain \cite{kh} will be called $b$-metric in this paper. Now we like to recall the definition of a $b$-metric space from \cite{kh} as follows:

\begin{definition} [\cite{kh}, Definition 6]
Let $X$ be a non-empty set and $s>0.$ A distance function $\rho: X \times X\rightarrow [0,\infty)$ is said to be a $b$-metric on $X$ if it satisfies the following conditions :
\begin{enumerate}
\item[(i)] $\rho(x,y)=0\Longleftrightarrow x=y~\forall~(x,y)\in X \times X$.

\item[(ii)] $\rho(x,y)=\rho(y,x)~$ $\forall~ (x,y)\in X \times X$.

\item[(iii)]$\rho(x,z)\leq s[\rho(x,y)+\rho(y,z)]~$ $\forall~ x,y,z\in X$.
\end{enumerate}
Then the triple $(X,\rho,s)$ is called a $b$-metric space.
\end{definition}
If we take $s=1$ then $X$ will be a metric space. So $b$-metric space is weaker than metric space. In the year 2013, Khojasteh et al. \cite{ks} introduced the notion of a $\theta$-metric space by introducing the concept of an $B$-action on the set $[0,\infty)\times [0,\infty).$ Before proceeding to the definition of $\theta$-metric space, we first recall the definition of an $B$-action from \cite{ks} as follows.

\begin{definition}[\cite{ks}, Definition 4]
Let $\theta:[0,\infty)\times [0,\infty)\rightarrow [0,\infty)$ be a continuous mapping with respect to each variable. Let $Im(\theta)=\{\theta(s,t): s,t\geq 0\}.$ Then $\theta$ is called an $B$-action iff the following conditions are satisfied :
\begin{enumerate}
\item[(i)] $\theta(0,0)=0$ and $\theta(s,t)=\theta(t,s) ~\forall~ s,t \geq 0$.

\item[(ii)] $\theta(x,y)<\theta(s,t)$ if either $x\leq s,~ y<t$ or $x< s,~ y\leq t$.

\item[(iii)] For each $m\in Im(\theta)$ and for each $t\in [0,m],$ there exists $s\in [0,m]$ such that $\theta(s,t)=m.$

\item[(iv)] $\theta(s,0)\leq s ~\forall~ s>0$.

\end{enumerate}

\end{definition}
They denoted the collection of all $B$-actions by $Y$. Now we like to recall the definition of a $\theta$-metric space from \cite{ks} as follows :

\begin{definition}[\cite{ks}, Definition 11]
Let $X$ be a non-empty set. A distance function $\rho: X \times X\rightarrow [0,\infty)$ is said to be a $\theta$-metric on $X$ with respect to an $B$-action $\theta\in Y$  if the following conditions are satisfied :
\begin{enumerate}
\item[(i)] $\rho(x,y)=0\Longleftrightarrow x=y~\forall~(x,y)\in X \times X$.

\item[(ii)] $\rho(x,y)=\rho(y,x)~$ $\forall~ (x,y)\in X \times X$.

\item[(iii)]$\rho(x,z)\leq \theta(\rho(x,y),\rho(y,z))~$ $\forall~ x,y,z\in X$.

\end{enumerate}
Then the triple $(X,\rho, \theta)$ is called a $\theta$-metric space. If we take $\theta(s,t)=s+t, ~s,t\geq 0$ then $\theta$-metric space reduce to metric space. But one thing to notice is that, in this case, the $B$-action $\theta$ is continuous at the point $(0,0).$

\end{definition}
In \cite{kh}, Khamsi and Hussain defined a natural topology on $b$-metric spaces, similarly to that of metric spaces and study the topological properties and KKM mappings on such kind of spaces. In \cite{ks}, Khojasteh et al. defined a natural topology on $\theta$-metric space and study the topological properties.

In this manuscript, we use Chittenden's metrization theorem \cite{CD} to prove the metrizability of $b$-metric spaces and $\theta$-metric spaces. Further, for more details one can see \cite{FR}. Chittenden's metrization theorem came up in the year 1917, which give sufficient conditions to show that a topological space $X$ is metrizable.

Before proceeding to our main result, we first recall the metrization result due to Chittenden \cite{CD}. Let $X$ be a topological space and $F:X \times X\rightarrow [0,\infty)$ be a distance function on $X$. If the distance function $F$ satisfies the following conditions:
\begin{enumerate}
\item[(i)] $F(x,y)=0\Longleftrightarrow x=y~\forall~(x,y)\in X \times X$.

\item[(ii)] $F(x,y)=F(y,x),~$ $\forall~ (x,y)\in X \times X$.

\item[(iii)](Uniformly regular) For every $\varepsilon>0$ and $x,y,z \in X$ there exists $\phi(\varepsilon)>0$ such that if $F(x,y)<\phi(\varepsilon)$ and $F(y,z)<\phi(\varepsilon)$ then $F(x,z)<\varepsilon$
\end{enumerate}
then the topological space $X$ is metrizable. 

\section{\bf Main Results}
In \cite{an}, Tran Van An et al. provide a sufficient condition for $b$-metric space to be metrizable. First of all we like to state the theorem due to Tran Van An et al. 

\begin{theorem} [\cite{an}, Theorem 3.15.]
Let $(X,D,K)$ be a $b$-metric space. If $D$ is continuous in one variable then every open cover of $X$ has an open refinement which is both locally finite and $\sigma$-discrete.
\end{theorem}

\begin{corollary} [\cite{an}, Corollary 3.17.]
Let $(X,D,K)$ be a $b$-metric space. If $D$ is continuous in one variable then $X$ is metrizable.
\end{corollary}

They proved the metrizability by assuming that the distance function on the space is continuous in one variable but in this section, we will provide a short proof of the metrizability of $b$-metric space without any assumption on the distance function.

\begin{theorem} \label{FMS}
Let $(X,d,S), S>0$ be a $b$-metric space. Then $X$ is metrizable.
\end{theorem}

\begin{proof}
 Let $(X,d,S)$ be a $b$-metric space. By the definition of a $b$-metric space, the distance function $d:X \times X\rightarrow [0,\infty)$ on $X$ satisfies the first two conditions of Chittenden's metrization result i.e,
\begin{enumerate}
\item[(i)] $d(x,y)=0\Longleftrightarrow x=y~\forall~(x,y)\in X \times X$;

\item[(ii)] $d(x,y)=d(y,x)~$ $\forall~ (x,y)\in X \times X$.
\end{enumerate}
Now we will prove the third condition i.e., the `uniformly regular' condition. Let $\varepsilon>0$ and $x,y,z \in X.$ If $x=z$, then $d(x,z)=0$. So in this case $\phi(\varepsilon)=r$ where $r$ is any positive real number will serve the purpose. Now let $x\neq z.$ Then $d(x,z)>0.$ So by the definition of a $b$-metric space we have,
\begin{equation}
d(x,z)\leq S[d(x,y)+ d(y,z)]
\label{aa}
\end{equation}
Now let us choose $\phi(\varepsilon)=\frac{\varepsilon}{2S}.$ If $d(x,y)< \frac{\varepsilon}{2S}$ and $d(y,z)<\frac{\varepsilon}{2S}$ then $d(x,y)+d(y,z)<\frac{\varepsilon}{S}.$ So by the equation \ref{aa}, we have 
$$ d(x,z)<S. \frac{\varepsilon}{S}$$
$$\Rightarrow d(x,z)<\varepsilon.$$
This shows that the distance function $d$ of a $b$-metric space satisfies the uniformly regular condition. Consequently, by Chittenden's metrization result we can conclude that the $b$-metric space $X$ is metrizable.
\end{proof}

\begin{remark}
So from Theorem \ref{FMS} we can conclude that if $(X,D,S), S>0$ be a $b$-metric space then there exists a metric $d:X \times X\rightarrow [0,\infty)$ on $X$ such that $X$ is metrizable with respect to $d.$ So we can conclude that the topological properties of $b$-metric spaces discussed in \cite[\, Proposition 2, Proposition 3]{kh} are equivalent to those of the standard metric spaces.
\end{remark}

Now in our next theorem we prove the metrizability of $\theta$-metric space through Chittenden's metrization theorem.

\begin{theorem}\label{met}
Let $(X,d,\theta)$ be a $\theta$-metric space where $\theta$ is an $B$-action on $[0,\infty)\times [0,\infty).$ Then $X$ is metrizable.
\end{theorem}

\begin{proof}
First of all, we show that the $B$-action $\theta$ is continuous at the point $(0,0).$ Suppose that $\{(s_n,t_n)\}_{n\in \mathbb{N}}$ be a sequence in $[0,\infty)\times [0,\infty)$ such that $(s_n,t_n)\rightarrow (0,0)$ as $n\rightarrow \infty.$ This implies $s_{n}\rightarrow 0$ and $t_{n}\rightarrow 0$ as $n\rightarrow \infty$ in the standard norm in $[0,\infty)\times [0,\infty).$ Now as the $B$-action $\theta$ is continuous in the second variable so keeping the first coordinate fixed and move the second coordinate as $n\rightarrow \infty$ we have, $\theta(s_n,t_n)\rightarrow \theta(s_n,0)$ as $n\rightarrow \infty.$ Let $\varepsilon>0.$ Now for $\frac{\varepsilon}{2}>0$ there exists $k_1\in \mathbb{N}$ such that $$\vert \theta(s_n,t_n)- \theta(s_n,0)\vert < \frac{\varepsilon}{2}~\forall~n\geq k_1.$$ Also since $s_n \rightarrow 0$ as $n\rightarrow \infty,$ so, for $\frac{\varepsilon}{2}>0$ there exists $k_2 \in \mathbb{N}$ such that $$s_n< \frac{\varepsilon}{2}~\forall~n\geq k_2.$$ Take $k=\mbox{max}\{k_1,k_2\}.$ Then for $n\geq k$ we have,
$$\vert \theta(s_n,t_n)\vert = \vert \theta(s_n,t_n)- \theta(s_n,0)+\theta(s_n,0)\vert\leq \vert \theta(s_n,t_n)- \theta(s_n,0)\vert + \theta(s_n,0)$$
$$\Longrightarrow \vert \theta(s_n,t_n)\vert \leq \vert \theta(s_n,t_n)- \theta(s_n,0)\vert + s_{n}$$
$$\Longrightarrow \theta(s_n,t_n)<\varepsilon.$$
This shows that the $B$-action $\theta$ is continuous at the point $(0,0).$ Now we will prove that $X$ is metrizable.

By the definition of a $\theta$-metric space, the distance function $d:X \times X\rightarrow [0,\infty)$ on $X$ satisfies the first two conditions of Chittenden's metrization result i.e,
\begin{enumerate}
\item[(i)] $d(x,y)=0\Longleftrightarrow x=y~\forall~(x,y)\in X \times X$;

\item[(ii)] $d(x,y)=d(y,x)~$ $\forall~ (x,y)\in X \times X$.
\end{enumerate}
Now we will prove the third condition i.e., the `uniformly regular' condition. Let $\varepsilon>0$ and $x,y,z \in X.$ If $x=z$, then $d(x,z)=0$. So in this case $\phi(\varepsilon)=r$ where $r$ is any positive real number will serve the purpose. Now let $x\neq z.$ Then $d(x,z)>0.$ So by the definition of a $\theta$-metric space we have,
\begin{equation}
d(x,z)\leq \theta(d(x,y), d(y,z)).
\label{ab}
\end{equation}
Since the $B$-action is continuous at $(0,0)$ so for $\varepsilon>0$ there exists $\delta>0$ such that $$|\theta(s,t)-\theta(0,0)|<\varepsilon~\mbox{whenever}~(s,t)\in B\Big((0,0),\delta \Big)\bigcap \Big([0,\infty)\times [0,\infty) \Big)$$
$$\Rightarrow \theta(s,t)< \varepsilon ~\mbox{whenever}~(s,t)\in B\Big((0,0),\delta\Big)\bigcap \Big([0,\infty)\times [0,\infty)\Big)$$
Here $B\Big((0,0),\delta\Big)$ denotes the open ball centered at $(0,0)$ and radius $\delta$ in the standard norm.
Let us choose $\phi(\varepsilon)=\frac{\delta}{\sqrt{2}}.$ Now 
$$d(x,y)<\frac{\delta}{\sqrt{2}}, d(y,z)<\frac{\delta}{\sqrt{2}}\Longrightarrow (d(x,y),d(y,z))\in B\Big((0,0),\delta\Big)\bigcap \Big([0,\infty)\times [0,\infty)\Big)$$
$$\Longrightarrow \theta(d(x,y),d(y,z))<\varepsilon.$$ So from the equation \ref{ab} we have, $$d(x,z)<\varepsilon.$$ This shows that the distance function $d$ of a $\theta$-metric space satisfies the uniformly regular condition. Consequently, by Chittenden's metrization result we can conclude that the $\theta$-metric space $X$ is metrizable.


\end{proof}



\bibliographystyle{plain}

\end{document}